\documentclass[11pt,letterpaper]{article}
\usepackage[margin=30mm]{geometry}
\usepackage{amsmath,amssymb,amsthm,mathtools}
\usepackage{microtype}
\usepackage{xcolor}
\usepackage{tikz}
\usepackage{float}
\usepackage[hidelinks]{hyperref}
\usetikzlibrary{arrows.meta,calc}

\definecolor{pairblue}{RGB}{27,94,158}
\definecolor{pairred}{RGB}{190,52,52}
\definecolor{softgray}{RGB}{105,105,105}

\tikzset{
  terminal/.style={
    circle,fill=white,inner sep=1.5pt,line width=0.85pt
  },
  pair one/.style={draw=pairblue,line width=1.2pt},
  pair two/.style={draw=pairred,line width=1.2pt},
  panel/.style={
    rounded corners=2pt,draw=black!18,fill=black!1
  },
  panel caption/.style={
    font=\scriptsize,align=center,text=black!78
  }
}

\newtheorem{theorem}{Theorem}
\newtheorem{lemma}{Lemma}

\newcommand{\R}{\mathbb{R}}
\newcommand{\len}{\operatorname{len}}

\title{Joining Two Pairs of Planar Points by Disjoint Arcs: The Sharp $\sqrt{2}$ Length Bound}
\author{\small George M. Georgiou\thanks{\small School of Computer Science and Engineering, California State University, San Bernardino}\\[3pt]
\small \href{mailto:georgiou@csusb.edu}{georgiou@csusb.edu}}
\date{}

\begin{document}
\maketitle

\begin{abstract}
Problem F16 of Croft, Falconer, and Guy asks for the least worst-case
length needed to join prescribed pairs of points by pairwise disjoint planar
arcs, when the distance within each pair is at most one.  The book suggests
that the answer for two pairs is $\sqrt2$.  We prove this exactly.  The lower
bound is a crossing argument in a square.  The upper bound follows from a
sharp ellipse lemma.  As a consequence, the value proposed in the book for
three pairs, $(\sqrt3+1)/2$, cannot be correct under the literal formulation,
because the constants are nondecreasing in the number of pairs.
\end{abstract}

\section{Definition and result}

Let
\[
 \mathcal P=((x_1,y_1),(x_2,y_2))
\]
be two pairs of distinct planar points, with the four terminals distinct and
$|x_i-y_i|\leq 1$.  Throughout, $|uv|$ denotes the Euclidean distance
between points $u,v\in\R^2$, and $\len(\gamma)$ denotes the length of a
rectifiable arc $\gamma$.  Put
\[
 \lambda(\mathcal P)=
 \inf_{\gamma_1,\gamma_2}
 \max\{\len(\gamma_1),\len(\gamma_2)\},
\]
where $\gamma_i$ is a rectifiable simple arc from $x_i$ to $y_i$ and
$\gamma_1\cap\gamma_2=\varnothing$.  Define the two-pair constant by
\[
 C_2=\sup_{\mathcal P}\lambda(\mathcal P).
\]
The use of an infimum is natural: in an extremal configuration, the shortest
route may pass through a terminal of the other pair and can only be
approached by disjoint arcs.

\begin{theorem}
The two-pair constant in Problem F16 is
\[
 C_2=\sqrt2.
\]
\end{theorem}

The proof uses the following elementary lemma.

\begin{lemma}[ellipse lemma]\label{lem:ellipse}
Let two segments $AB$ and $CD$, each of length at most one, cross at an
interior point $O$.  Put
\[
 a=|AO|,\quad b=|BO|,\quad c=|CO|,\quad d=|DO|.
\]
Relabel the endpoints, and interchange the two pairs if necessary, so that
\[
 a\leq b,\qquad c\leq d,\qquad a\leq c.
\]
Then
\[
 |CA|+|AD|\leq\sqrt2.
\]
\end{lemma}

\begin{proof}
The ordering assumptions and the unit-length hypotheses imply
$a\leq c\leq\tfrac12$.  Write $q=\tfrac12-c$, so $q\geq0$ and
$a+q\leq\tfrac12$.  Extend $D$ along the ray $OD$ to a point $D'$ such that
$|CD'|=1$.  It is enough to prove
\[
 |CA|+|AD'|\leq\sqrt2.
\]
Indeed, $d\geq c\geq a$.  If $\theta=\angle AOD$, then the squared distance
from $A$ to the point on the ray $OD$ at distance $s$ from $O$ is
\[
 a^2+s^2-2as\cos\theta.
\]
For $s\geq a$, its derivative is
$2(s-a\cos\theta)\geq2(s-a)\geq0$.  Thus the distance from $A$ is
nondecreasing there, and $|AD|\leq|AD'|$.

The ellipse with foci $C,D'$ and focal distance $|CD'|=1$ whose major-axis
length is $\sqrt2$ has semiaxes $1/\sqrt2$ and $1/2$.  Use orthonormal
coordinates centered at the midpoint of $CD'$, with the positive first axis
pointing toward $C$; orient the second axis so that $A$ has nonnegative
second coordinate.  Then
\[
 C=(\tfrac12,0),\qquad D'=(-\tfrac12,0),\qquad O=(q,0).
\]
If $t=\cos\angle AOC$, the coordinates of $A$ are therefore
\[
 (q+at,\;a\sqrt{1-t^2}).
\]
The filled ellipse has equation $2x^2+4y^2\leq1$.  At $A$, its left-hand
side is
\begin{align*}
 2(q+at)^2+4a^2(1-t^2)
 &=4(a^2+q^2)-2(at-q)^2 \\
 &\leq 4(a^2+q^2) \\
 &\leq 4(a+q)^2 \\
 &\leq 1.
\end{align*}
Here the penultimate inequality uses $a,q\geq0$, and the last uses
$a+q=a+\tfrac12-c\leq\tfrac12$.  Hence $A$ lies in the ellipse, so
\[
 |CA|+|AD'|\leq\sqrt2.
\]
Together with $|AD|\leq|AD'|$, this proves the claim.
\end{proof}

The bound is sharp: equality holds when the two segments are perpendicular
unit segments bisected by $O$.

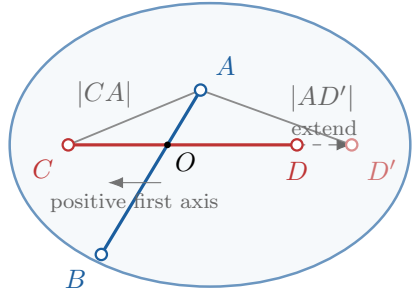
\begin{figure}[H]
\centering
\begin{tikzpicture}[
  x=1.25cm,y=1.25cm,font=\small,line cap=round,line join=round
]
  \fill[pairblue!4] (0,0) ellipse (2.12 and 1.50);
  \draw[pairblue!65,line width=0.95pt] (0,0) ellipse (2.12 and 1.50);
  \node[pairblue!80!black,fill=white,inner sep=1pt] at (0,1.90)
    {$E=\{X:|CX|+|XD'|\leq\sqrt2\}$};

  \coordinate (C) at (-1.50,0);
  \coordinate (Dp) at (1.50,0);
  \coordinate (D) at (0.92,0);
  \coordinate (O) at (-0.45,0);
  \coordinate (A) at (-0.10,0.58);
  \coordinate (B) at (-1.15,-1.16);

  \draw[black!45,line width=0.7pt] (C)--(A)--(Dp);
  \node[black!60,above left=-1pt] at ($(C)!0.54!(A)$) {$|CA|$};
  \node[black!60,above right=-1pt] at ($(A)!0.54!(Dp)$) {$|AD'|$};
  \draw[pair one] (A)--(B);
  \draw[pair two] (C)--(D);
  \draw[softgray,dashed,-{Latex[length=2mm]}] (D)--(Dp);
  \node[softgray,font=\scriptsize,above] at ($(D)!0.5!(Dp)$)
    {extend};

  \node[terminal,draw=pairblue,label={[pairblue]above right:$A$}] at (A) {};
  \node[terminal,draw=pairblue,label={[pairblue]below left:$B$}] at (B) {};
  \node[terminal,draw=pairred,label={[pairred]below left:$C$}] at (C) {};
  \node[terminal,draw=pairred,label={[pairred]below:$D$}] at (D) {};
  \node[terminal,draw=pairred!55,label={[pairred!70]below right:$D'$}]
    at (Dp) {};
  \fill[black] (O) circle (1.25pt);
  \node[below right=-1pt] at (O) {$O$};

  \draw[-{Latex[length=2mm]},black!55] (-0.52,-0.40)--(-1.08,-0.40);
  \node[black!60,font=\scriptsize,below] at (-0.80,-0.40)
    {positive first axis};
\end{tikzpicture}
\caption{Geometry of the ellipse lemma.  Extending $D$ to $D'$ fixes the
focal distance at one.  The calculation proves $A\in E$, which is precisely
$|CA|+|AD'|\leq\sqrt2$.  The coordinate axis is oriented toward $C$.}
\label{fig:ellipse}
\end{figure}

\section{Proof of the theorem}

\subsection*{Upper bound}

Consider any two terminal pairs $A,B$ and $C,D$, with
$|AB|,|CD|\leq1$.  If the straight segments $AB$ and $CD$ are disjoint,
they already give admissible arcs of length at most one.

Suppose first that the segments cross transversely at an interior point $O$.
Choose $A$ nearer to $O$ than $B$, and $C$ nearer to $O$ than $D$.  Then
\[
 |AO|\leq\frac{|AB|}{2}\leq\frac12,
 \qquad
 |CO|\leq\frac{|CD|}{2}\leq\frac12.
\]
After interchanging the two pairs if necessary, assume $|AO|\leq|CO|$.
This interchange merely swaps the two entries in the maximum defining
$\lambda(\mathcal P)$, so it does not change its value.
Lemma~\ref{lem:ellipse} yields
\[
 |CA|+|AD|\leq\sqrt2.
\]
Keep $AB$ as the first arc.  The polygonal chain $C\to A\to D$ meets $AB$
only at $A$: its two open segments lie on opposite sides of the line through
$AB$.  In a disk $B_\rho(A)$, delete the two short pieces incident with $A$
and join their new endpoints by a circular arc that avoids the radial segment
$AB$.  The replacement arc has length at most $2\pi\rho$, so the resulting
$C$--$D$ arc has length at most
\[
 |CA|+|AD|+2\pi\rho.
\]
It is simple and disjoint from $AB$.  Given $\varepsilon>0$, choosing
$\rho<\varepsilon/(2\pi)$ therefore gives an arc of length less than
$\sqrt2+\varepsilon$.  Hence
\[
 \lambda(\mathcal P)\leq\sqrt2.
\]

The remaining incidences are elementary limiting cases.  If one terminal
lies in the interior of the other segment, detour the latter segment around
that terminal inside a disk of radius $\rho$; the added length is
$O(\rho)$.  Since the four terminals are distinct, if $\eta$ is their
minimum pairwise distance, all terminal disks may be chosen with radius less
than $\eta/4$ and are then pairwise disjoint.  If the two segments overlap
collinearly, route one pair at distance $\rho$ on one side of their common
line and the other pair at distance $\rho$ on the other side, joining each
route to its endpoints by perpendicular pieces.  The two arcs are disjoint
and have lengths at most $|AB|+2\rho$ and $|CD|+2\rho$.  Letting
$\rho\downarrow0$ and using the infimum in the definition of $\lambda$
handles every degenerate case.  Therefore $C_2\leq\sqrt2$.

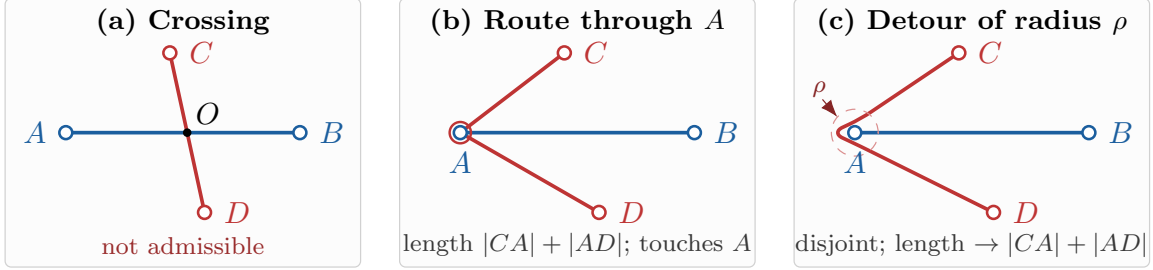
\begin{figure}[H]
\centering
\resizebox{0.98\textwidth}{!}{%
\begin{tikzpicture}[font=\small,line cap=round,line join=round]
  \begin{scope}[xshift=0cm]
    \draw[panel] (-2.05,-1.55) rectangle (2.05,1.55);
    \node[font=\bfseries\footnotesize] at (0,1.28) {(a) Crossing};
    \coordinate (A) at (-1.35,0);
    \coordinate (B) at (1.35,0);
    \coordinate (C) at (-0.15,0.92);
    \coordinate (D) at (0.25,-0.92);
    \draw[pair one] (A)--(B);
    \draw[pair two] (C)--(D);
    \fill[black] (0.05,0) circle (1.4pt);
    \node[above right=-1pt] at (0.05,0) {$O$};
    \node[terminal,draw=pairblue,label={[pairblue]left:$A$}] at (A) {};
    \node[terminal,draw=pairblue,label={[pairblue]right:$B$}] at (B) {};
    \node[terminal,draw=pairred,label={[pairred]right:$C$}] at (C) {};
    \node[terminal,draw=pairred,label={[pairred]right:$D$}] at (D) {};
    \node[panel caption,pairred!80!black] at (0,-1.30)
      {not admissible};
  \end{scope}

  \begin{scope}[xshift=4.55cm]
    \draw[panel] (-2.05,-1.55) rectangle (2.05,1.55);
    \node[font=\bfseries\footnotesize] at (0,1.28) {(b) Route through $A$};
    \coordinate (A) at (-1.35,0);
    \coordinate (B) at (1.35,0);
    \coordinate (C) at (-0.15,0.92);
    \coordinate (D) at (0.25,-0.92);
    \draw[pair one] (A)--(B);
    \draw[pair two] (C)--(A)--(D);
    \node[terminal,draw=pairblue,label={[pairblue]below:$A$}] at (A) {};
    \node[terminal,draw=pairblue,label={[pairblue]right:$B$}] at (B) {};
    \node[terminal,draw=pairred,label={[pairred]right:$C$}] at (C) {};
    \node[terminal,draw=pairred,label={[pairred]right:$D$}] at (D) {};
    \draw[pairred,line width=0.8pt] (A) circle (3.5pt);
    \node[panel caption] at (0,-1.30)
      {length $|CA|+|AD|$; touches $A$};
  \end{scope}

  \begin{scope}[xshift=9.10cm]
    \draw[panel] (-2.05,-1.55) rectangle (2.05,1.55);
    \node[font=\bfseries\footnotesize] at (0,1.28)
      {(c) Detour of radius $\rho$};
    \coordinate (A) at (-1.35,0);
    \coordinate (B) at (1.35,0);
    \coordinate (C) at (-0.15,0.92);
    \coordinate (D) at (0.25,-0.92);
    \draw[pair one] (A)--(B);
    \draw[pair two,rounded corners=4pt]
      (C)--(-1.28,0.16)--(-1.62,0)--(-1.28,-0.16)--(D);
    \draw[pairred!55,dashed] (A) circle (0.27);
    \draw[-{Latex[length=2mm]},pairred!70!black]
      (-1.72,0.38)--(-1.55,0.18);
    \node[pairred!70!black,font=\scriptsize] at (-1.76,0.50) {$\rho$};
    \node[terminal,draw=pairblue,label={[pairblue]below:$A$}] at (A) {};
    \node[terminal,draw=pairblue,label={[pairblue]right:$B$}] at (B) {};
    \node[terminal,draw=pairred,label={[pairred]right:$C$}] at (C) {};
    \node[terminal,draw=pairred,label={[pairred]right:$D$}] at (D) {};
    \node[panel caption] at (0,-1.30)
      {disjoint; length $\to |CA|+|AD|$};
  \end{scope}
\end{tikzpicture}%
}
\caption{The upper-bound construction.  The limiting route through $A$ has
the desired length but touches the blue arc.  A radius-$\rho$ detour makes the
arcs disjoint, with an error tending to zero.}
\label{fig:detour}
\end{figure}

\subsection*{Lower bound}

Take the four terminals
\[
 A=(-\tfrac12,0),\quad B=(\tfrac12,0),\quad
 C=(0,-\tfrac12),\quad D=(0,\tfrac12).
\]
Both prescribed distances equal one.  Let $\alpha$ join $A$ to $B$ and
$\beta$ join $C$ to $D$.

Assume, for a contradiction, that $\alpha$ and $\beta$ are disjoint and that
both have length less than $\sqrt2$.  For every $P=(x,y)$ on $\alpha$, the
two subarcs cut at $P$ give
\[
 \len(\alpha)\geq |A-P|+|P-B|.
\]
For fixed $y$, the right-hand side is a convex function of $x$ symmetric
about $x=0$, and is therefore minimized at $x=0$.  There it equals
\[
 2\sqrt{\frac14+y^2}=\sqrt{1+4y^2}.
\]
It follows that every point of $\alpha$ has $|y|<\tfrac12$.  Similarly,
every point of $\beta$ has $|x|<\tfrac12$.

Let $Q=[-\tfrac12,\tfrac12]^2$.  By taking the portion of $\alpha$ between
its last visit to the left side of $Q$ before its first subsequent visit to the
right side, we obtain a subarc in $Q$ joining the left and right sides and
otherwise lying in the interior of $Q$.  Analogously, $\beta$ contains a
subarc in $Q$ joining the bottom and top sides.  The planar crossing lemma
says that such two subarcs must intersect.  For completeness, join the
left--right subarc to the boundary arc in $\partial Q$ that passes through
the bottom side.  The result is a Jordan curve.  Immediately after leaving its bottom endpoint,
the bottom--top subarc is inside this curve; just before reaching its top
endpoint, it is outside.  Since it is otherwise in the interior of $Q$, it
cannot cross the added boundary portion and must cross the left--right
subarc.  This contradicts the disjointness of $\alpha$ and $\beta$.

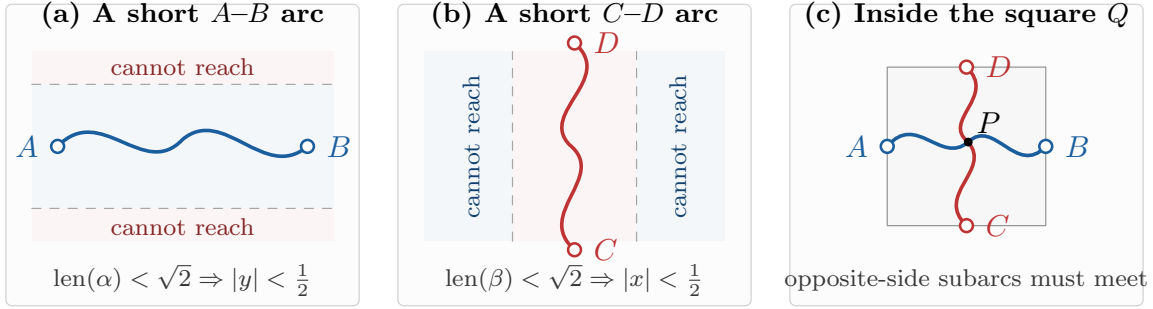
\begin{figure}[H]
\centering
\resizebox{0.98\textwidth}{!}{%
\begin{tikzpicture}[font=\small,line cap=round,line join=round]
  \begin{scope}[xshift=0cm]
    \draw[panel] (-2.05,-1.85) rectangle (2.05,1.65);
    \node[font=\bfseries\footnotesize] at (0,1.52)
      {(a) A short $A$--$B$ arc};
    \fill[pairblue!5] (-1.75,-0.72) rectangle (1.75,0.72);
    \fill[pairred!5] (-1.75,0.72) rectangle (1.75,1.10);
    \fill[pairred!5] (-1.75,-1.10) rectangle (1.75,-0.72);
    \draw[black!35,dashed] (-1.75,0.72)--(1.75,0.72);
    \draw[black!35,dashed] (-1.75,-0.72)--(1.75,-0.72);
    \coordinate (A) at (-1.45,0);
    \coordinate (B) at (1.45,0);
    \draw[pair one]
      (A)..controls(-0.95,0.52) and (-0.45,-0.40)..(0,0.08)
         ..controls(0.48,0.48) and (0.95,-0.42)..(B);
    \node[terminal,draw=pairblue,label={[pairblue]left:$A$}] at (A) {};
    \node[terminal,draw=pairblue,label={[pairblue]right:$B$}] at (B) {};
    \node[panel caption] at (0,-1.55)
      {$\len(\alpha)<\sqrt2\Rightarrow |y|<\tfrac12$};
    \node[font=\scriptsize,pairred!70!black] at (0,0.93) {cannot reach};
    \node[font=\scriptsize,pairred!70!black] at (0,-0.93) {cannot reach};
  \end{scope}

  \begin{scope}[xshift=4.55cm]
    \draw[panel] (-2.05,-1.85) rectangle (2.05,1.65);
    \node[font=\bfseries\footnotesize] at (0,1.52)
      {(b) A short $C$--$D$ arc};
    \fill[pairred!5] (-0.72,-1.10) rectangle (0.72,1.10);
    \fill[pairblue!5] (-1.75,-1.10) rectangle (-0.72,1.10);
    \fill[pairblue!5] (0.72,-1.10) rectangle (1.75,1.10);
    \draw[black!35,dashed] (-0.72,-1.10)--(-0.72,1.10);
    \draw[black!35,dashed] (0.72,-1.10)--(0.72,1.10);
    \coordinate (C) at (0,-1.20);
    \coordinate (D) at (0,1.20);
    \draw[pair two]
      (C)..controls(-0.48,-0.78) and (0.40,-0.35)..(-0.06,0.02)
         ..controls(-0.42,0.40) and (0.48,0.78)..(D);
    \node[terminal,draw=pairred,label={[pairred]right:$C$}] at (C) {};
    \node[terminal,draw=pairred,label={[pairred]right:$D$}] at (D) {};
    \node[panel caption] at (0,-1.55)
      {$\len(\beta)<\sqrt2\Rightarrow |x|<\tfrac12$};
    \node[font=\scriptsize,pairblue!70!black,rotate=90] at (-1.20,0)
      {cannot reach};
    \node[font=\scriptsize,pairblue!70!black,rotate=90] at (1.20,0)
      {cannot reach};
  \end{scope}

  \begin{scope}[xshift=9.10cm]
    \draw[panel] (-2.05,-1.85) rectangle (2.05,1.65);
    \node[font=\bfseries\footnotesize] at (0,1.52)
      {(c) Inside the square $Q$};
    \fill[black!3] (-0.92,-0.92) rectangle (0.92,0.92);
    \draw[black!42] (-0.92,-0.92) rectangle (0.92,0.92);
    \coordinate (A) at (-0.92,0);
    \coordinate (B) at (0.92,0);
    \coordinate (C) at (0,-0.92);
    \coordinate (D) at (0,0.92);
    \coordinate (P) at (0.02,0.05);
    \draw[pair one]
      (A)..controls(-0.58,0.40) and (-0.30,-0.22)..(P)
         ..controls(0.32,0.34) and (0.58,-0.36)..(B);
    \draw[pair two]
      (C)..controls(-0.38,-0.56) and (0.30,-0.28)..(P)
         ..controls(-0.30,0.30) and (0.36,0.58)..(D);
    \fill[black] (P) circle (1.45pt);
    \node[above right=-1pt] at (P) {$P$};
    \node[terminal,draw=pairblue,label={[pairblue]left:$A$}] at (A) {};
    \node[terminal,draw=pairblue,label={[pairblue]right:$B$}] at (B) {};
    \node[terminal,draw=pairred,label={[pairred]right:$C$}] at (C) {};
    \node[terminal,draw=pairred,label={[pairred]right:$D$}] at (D) {};
    \node[panel caption] at (0,-1.55)
      {opposite-side subarcs must meet};
  \end{scope}
\end{tikzpicture}%
}
\caption{The lower-bound mechanism.  Each path shorter than $\sqrt2$ is
trapped in a strip.  Inside $Q$, the resulting opposite-side subarcs must
intersect.}
\label{fig:lower-bound}
\end{figure}

Thus at least one of the two arcs has length at least $\sqrt2$, so the displayed
configuration has $\lambda(\mathcal P)\geq\sqrt2$.  Hence
$C_2\geq\sqrt2$, completing the proof.

\section{The three-pair value quoted in the book}

Let $C_n$ denote the corresponding constant for $n$ pairs.  The sequence is
nondecreasing.  Starting with any configuration of $n$ pairs, add one more
pair in a small disk far from all existing terminals; its two points can be
joined inside that disk.  Every admissible family for the enlarged
configuration restricts to an admissible family for the original one, so
the enlarged configuration has cost at least that of the original.  Taking
suprema gives
\[
 C_{n+1}\geq C_n.
\]
In particular,
\[
 C_3\geq C_2=\sqrt2.
\]
But
\[
 \frac{\sqrt3+1}{2}<\sqrt2,
\]
since $(\sqrt3+1)^2=4+2\sqrt3<8$.  Thus the value
$(\sqrt3+1)/2$ proposed for $n=3$ in F16 is incompatible with the exact
two-pair result under the stated definition, independently of the eventual
value of $C_3$.  Determining $C_3$ remains a separate problem.
\section*{Acknowledgments}

The author used OpenAI’s ChatGPT 5.6 Sol to assist with the initial discovery of the result and drafting of the proof, and Google’s Gemini 3.1 Pro for subsequent review and refinement. The author independently verified the final arguments and assumes responsibility for their correctness.

\end{document}